\documentclass[12pt,reqno]{amsart}   %

\usepackage[T1]{fontenc}            %
\usepackage[utf8]{inputenc}         %
\usepackage{lmodern}                %

\usepackage[english]{babel}
\usepackage{microtype}              %
\usepackage{comment}

\usepackage{geometry}
\usepackage{amsmath, amssymb, amsfonts, amsthm}  %
\usepackage{mathtools}              %
\usepackage{mathrsfs}               %
\usepackage{bm}                     %
\usepackage{esint}

\usepackage{xcolor}                 %

\usepackage[colorlinks=true,        %
            linkcolor=blue,         %
            citecolor=blue,         %
            urlcolor=blue]{hyperref}

\usepackage[capitalise,nameinlink]{cleveref}
\crefname{equation}{}{} 

\usepackage{enumitem}
\setlist[enumerate,1]{label=\arabic*., ref=\arabic*}  %

\numberwithin{equation}{section}   

\theoremstyle{plain}
\newtheorem{theorem}{Theorem}[section]
\newtheorem{proposition}[theorem]{Proposition}
\newtheorem{lemma}[theorem]{Lemma}

\newtheorem{assumption}{Assumption}

\theoremstyle{definition}

\theoremstyle{remark}
\newtheorem{remark}[theorem]{Remark}

\makeatletter
\renewcommand\paragraph{%
  \@startsection{paragraph}{4}{0pt}%
    {3.25ex \@plus1ex \@minus.2ex}%
    {1.5ex  \@plus.2ex}%
    {\normalfont\normalsize\bfseries}%
}
\makeatother

\newcommand{\R}{\mathbb{R}}

\begin{document}

\title[Quantitative Scattering on Asymptotically Flat Spacetimes]{Quantitative Scattering for the Energy-Critical Wave Equation on Asymptotically Flat Spacetimes}
\author[B. Dodson]{Benjamin Dodson}
\author[S. Looi]{Sam Looi}

\date{\today}
\begin{abstract}
We prove quantitative scattering for the three-dimensional defocusing energy-critical quintic wave equation on a class of asymptotically flat, possibly non-stationary perturbations of Minkowski space, by establishing the first explicit global $L^8_{t,x}$ bound in this variable-coefficient setting. Earlier work in this setting proved scattering only qualitatively. For
\[
Pu=u^5,\qquad P=\partial_\alpha(g^{\alpha\beta}\partial_\beta),
\]
we show that, under smallness, decay, and regularity assumptions on the metric, and assuming a priori $\dot H^5\times\dot H^4$ and $L^2\times\dot H^{-1}$ bounds on the solution, the critical spacetime norm $\|u\|_{L^8_{t,x}(\mathbb R\times\mathbb R^3)}$ satisfies an explicit exponential-type estimate in terms of the energy and the a priori bound. This upgrades the qualitative scattering theory in this setting to a quantitative one.

The main difficulty is to control geometric error terms over long times. We handle this by splitting the Duhamel history into recent past and remote past. For the recent past, we prove a variable-coefficient interaction Morawetz estimate that yields, on every sufficiently long interval, a time at which the recent nonlinear contribution is small. For the remote past, we prove a dispersive estimate from integrated local energy decay together with a transfer of pointwise decay from large radius to large time. Combining these estimates gives the explicit global $L^8_{t,x}$ bound.
\end{abstract}

\maketitle

\section{Introduction}\label{sec:intro}

\subsection*{1.  Background}

We study the long-time dynamics of the defocusing, energy-critical quintic wave equation on a mildly perturbed Minkowski background,
\begin{equation}
    \label{eq:IVP}
    Pu=u^{5},\qquad P=\partial_\alpha \bigl(g^{\alpha\beta}\partial_\beta\bigr),\qquad u[0]\in\dot H^{1}\times L^{2},
\end{equation}
where $g^{\alpha\beta}(t,x)=m^{\alpha\beta}+h^{\alpha\beta}(t,x)$ is an  asymptotically flat, potentially non-stationary, perturbation of the Minkowski metric that satisfies the smallness and regularity assumptions \eqref{eq:hyp}. %

While global well-posedness and scattering are classical in the Minkowski setting ($h=0$), and quantitative bounds on the $L_{t,x}^8$ norm were recently established \cite{Dodson2024}, which built on the work of \cite{nakanishi1999a,tao2006}, the situation on asymptotically flat backgrounds has remained qualitative \cite{Looi_Tohaneanu_2021}. The transition to a quantitative theory has been obstructed by the difficulty of controlling the global error terms arising from the non-trivial spacetime geometry. This paper provides the first quantitative bound on the $L_{t,x}^8$ norm in this variable-coefficient setting. Our main theorem provides an explicit, exponential-type bound, thereby establishing quantitative scattering.

In the flat case $g=m$, the global Cauchy theory and scattering in the energy space are classical. The equation is energy critical because the conserved energy is invariant under the scaling $u(t,x)\mapsto \lambda^{1/2}u(\lambda t,\lambda x)$. Local well--posedness was proved by \cite{segal1963}. Scattering was proved in \cite{pecher1984} and \cite{rauch1981} for small energy initial data. For large, radially symmetric initial data, \cite{struwe1988} proved global well--posedness for $(\ref{eq:IVP})$ in the flat case for radially symmetric initial data. Later, \cite{ginibre1992} obtained a bound on the $L_{t,x}^{8}$ norm of the solution in the radially symmetric case, proving scattering. For non-symmetric initial data, \cite{grillakis1990} proved global well--posedness and persistence of regularity for smooth initial data. Later, \cite{shatah-struwe1994} proved global well-posedness for initial data in the energy space, and \cite{bahouri1999} proved a qualitative bound on the $L_{t,x}^{8}$ norm of a solution in the flat case. The paper \cite{nakanishi1999a} proved quantitative bounds for both the wave equation and the Klein--Gordon equation, see also \cite{tao2006}.

Building on the work of \cite{nakanishi1999a} and \cite{tao2006}, \cite{Dodson2024} obtained an explicit global bound for the critical Strichartz norm $\|u\|_{L^8_{t,x}}$ on Minkowski, depending on the conserved energy and an a priori high-regularity control. In this paper we extend this quantitative framework to the variable-coefficient setting. In comparison, for the asymptotically flat backgrounds of the type assumed here, \cite{Looi_Tohaneanu_2021} proved integrated local energy decay and qualitative scattering, while \cite{Looi2021} established sharp pointwise decay for the linear flow, but these results do not yield an explicit $L^8_{t,x}$ control. To supply such a bound, we construct an interaction Morawetz functional adapted to $P$, control the curvature-induced error terms with a variable-coefficient dispersive toolkit in the spirit of \cite{Looi2021}, along with the Strichartz estimates constructed by \cite{MetcalfeTataruParametrices}, and deduce an explicit, exponential-type bound for $\|u\|_{L^8_{t,x}}$ under \eqref{eq:hyp} and a priori high-regularity control, thereby upgrading qualitative scattering to a quantitative statement for a class of asymptotically flat spacetimes.

More precisely, \cite{Dodson2024} proved the explicit bound
\[
   \|u\|_{L^{8}_{t,x}(\mathbb R\times\mathbb R^{3})}
   \le  C(E)\exp\bigl(C\,E^{85/6+13/14}A^{11}\bigr),
\]
where $E$ is the conserved energy and $A$ measures higher regularity of the initial data.  
Such an estimate is useful in applications that require more than the mere finiteness of the Strichartz norm. The assumption utilized in \cite{Dodson2024} can be found in \cref{assumption_to_be_improved}.

The aim of this paper is to extend that quantitative theory to a class of \textbf{variable-coefficient} wave operators.
Throughout we assume
\begin{assumption}\label{assump:h}
We assume the metric perturbation $h$ satisfies the following conditions for sufficiently small $\varepsilon>0$, and some constants $\gamma>0, \delta>0$. Here $\underline{L} = \sum_{i = 1}^{3} \frac{x^{i}}{|x|} \partial_{x_{i}} - \partial_{t}$ denotes the incoming null vector field.
\begin{subequations}\label{eq:hyp}
\begin{align}
   |h(t,x)| &\le \varepsilon
               \frac{\langle t-|x|\rangle^{1/2}}
                    {\langle x\rangle^{\gamma}\langle t+|x|\rangle^{1/2}}, \label{eq:hyp-a}\\
   |h^{\underline L\underline L}(t,x)|
            &\le\varepsilon
               \frac{\langle t-|x|\rangle}
                    {\langle x\rangle^{\gamma}\langle t+|x|\rangle},\label{eq:hyp-b}\\
   |\partial^J h^{\mu\nu}(t,x)| &\le \varepsilon\,\langle x\rangle^{-1-\gamma} \qquad\text{for all } 1 \le |J|\le N+1=5,\label{eq:hyp-c} \\
    |h(t,x)|
            &\le 
               \langle x\rangle^{-3-\delta},\qquad \delta>0 \text{ arbitrarily small},\label{eq:hyp-d}
\end{align}
\end{subequations}
for some small $\gamma>0$ and $\varepsilon>0$. Here, $h^{\underline{L} \underline{L}} = h^{\alpha \beta} \underline{L}_{\alpha} \underline{L}_{\beta}$. We also assume
    \begin{equation}
        h^{0j}(t,x) = 0 \quad \text{for } j=1,2,3.
    \end{equation}
This assumption is only needed for the proof of the interaction Morawetz estimate (in \cref{sec:interaction_Morawetz}).
\end{assumption}
    
Assumption \cref{eq:hyp-d} is a short-range decay condition used only in the dispersive estimate (see \cref{sec:a_dispersive_estimate}). In contrast, \cref{eq:hyp-a,eq:hyp-b} allow a long-range tail for the metric $h$ (e.g. $\langle x\rangle^{-\gamma}$), while \cref{eq:hyp-c} imposes short-range decay on $\partial h$. 
Our assumptions, apart from the short-range condition \eqref{eq:hyp-d} required for the dispersive estimate, place our proof within the ``long-range metric, short-range derivatives'' framework of \cite{Looi_Tohaneanu_2021}, from which we derive the necessary integrated local energy decay estimates.

The collection of geometric hypotheses in this paper is tailored to the goal of proving quantitative scattering bounds, and it is therefore distinct from the structure used for the sharp pointwise decay estimates in \cite{Looi2021}. Specifically, where \cite{Looi2021} imposes uniform decay with full Klainerman vector-field control (referred to there as $h \in S^Z(\langle r\rangle^{-1-\sigma})$), this paper uses only control of standard partial derivatives, making it geometrically more permissive. Furthermore, while \cite{Looi2021} assumes local energy decay as a black-box hypothesis, we recall ILED under our assumptions (following \cite{Looi_Tohaneanu_2021} for order 0 ILED and \cite{Looi2022} for higher order ILED) and then build the quantitative theory. However, the present paper does not utilize the sharp pointwise decay results of \cite{Looi2021}. Only a certain ``$r$-to-$t$ decay transfer principle'', as explained in 
    As a result, this paper presents a self-contained and geometrically permissive set of assumptions for establishing quantitative scattering. It uses a less structurally demanding set of assumptions than those tailored for sharp pointwise decay, though the two sets of geometric conditions are distinct and not strictly comparable.

\cite{Looi_Tohaneanu_2021}'s main theorem shows that the unique global strong solution to \eqref{eq:IVP} exists and scatters in the energy space; in particular, $\lim_{t\to\pm\infty}\|u(t)\|_{L^6_x}=0$ and $u(t)$ approaches a linear solution in $\dot H^1_x \times L^2_x$. The argument is qualitative and driven by integrated local energy decay, and does not furnish an explicit a~priori, quantitative bound on the Strichartz norm. 

\subsection*{2. Main Result}

The primary contribution of this paper is to supply such a quantitative bound. To achieve this, we consider a class of asymptotically flat spacetimes where the metric perturbation $h^{\alpha\beta}$ satisfies the following set of smallness, regularity, and decay conditions. %

\begin{theorem}[Quantitative Scattering for the Energy-Critical Wave Equation on Asymptotically Flat Spacetimes]\label{thm:main}
Let $(\mathbb{R}^{1+3}, g)$ be a spacetime where the metric $g^{\alpha\beta} = m^{\alpha\beta} + h^{\alpha\beta}$ is a perturbation of the Minkowski metric $m$. Let $P = \partial_{\alpha}(g^{\alpha\beta}\partial_{\beta})$ be the associated d'Alembertian. We assume the metric perturbation $h$ satisfies \cref{assump:h}. Consider the defocusing energy-critical wave equation:
\[
Pu=u^{5},
\]
with initial data $(u_0, u_1) \in \dot{H}^1(\mathbb{R}^3) \times L^2(\mathbb{R}^3)$ satisfying the energy bound
\begin{equation}\label{1.5}
\int \frac{1}{2} |u_{1}|^{2} + \frac{1}{2} |\nabla u_{0}|^{2} + \frac{1}{6} |u_{0}|^{6} dx \leq E.
\end{equation}

Suppose the corresponding solution $u$ satisfies the following a priori high-regularity bound for some finite constant $A$:
\begin{equation}\label{assumption_to_be_improved}
\sup_{t \in \mathbb{R}} \, \big\|(u(t), u_{t}(t))\big\|_{\dot H^{5}(\mathbb{R}^3) \times \dot H^{4}(\mathbb{R}^3)} \leq A, \quad \sup_{t \in \mathbb{R}} \, \big\|(u(t), u_{t}(t))\big\|_{L^2(\mathbb{R}^3) \times \dot H^{-1}(\mathbb{R}^3)} \leq A.
\end{equation}

Then the solution $u$ is globally defined, belongs to the critical Strichartz space $L^8_{t,x}(\mathbb{R}\times\mathbb{R}^3)$, and satisfies the explicit quantitative spacetime bound
\[
\|u\|_{L^8_{t,x}(\mathbb{R}\times\mathbb{R}^3)} \leq C E^{4/7} A \exp\left(C E^{85/6} E^{13/14} A^{11}\right).
\]

Furthermore, the solution scatters to a solution of the linear equation $Pw=0$ as $t\to\pm\infty$.
\end{theorem}
The need for 5 spatial derivatives (and 4 for $u_t$) is explained in \cref{rem:derivative-bookkeeping}. 

As a notational convention, we write $S(t,s)$ for the propagator $S_g(t,s)$ from time $s$ to $t$, and abbreviate $S(t) := S_g(t) := S_g(t,0)$ when the initial time is $0$.

\subsection*{3.  Outline of the argument}

\subsubsection*{Outline of proof of \cref{thm:main}} We partition $\mathbb{R}$ into finitely many intervals $I$ such that $\| S(t)(u_{0}, u_{1}) \|_{L_{t,x}^{8}(I \times \mathbb{R}^{3})}$ is small. This is standard and proved in \cref{prop:finite_decomp}. %

If $|I|$ is not sufficiently large, then we just use $(\ref{assumption_to_be_improved})$ to obtain a bound on $\| u \|_{L_{t,x}^{8}(I \times \mathbb{R}^{3})}$. Indeed, by H{\"o}lder's inequality combined with the Sobolev embedding theorem, $\| u(t) \|_{L_{x}^{8}} \lesssim A$ for any $t \in \mathbb{R}$, so $\| u \|_{L_{t,x}^{8}(I \times \mathbb{R}^{3})} \lesssim |I|^{1/8} A$. 

For any sufficiently long interval from the decomposition provided in \cref{prop:finite_decomp}, \cref{prop:negligible_recent_nonlinear,prop:negligible_remote_past} are used in tandem to find a time $t_0$ where the Duhamel formula can be controlled by a smallness argument, thus bounding the $L_{t,x}^8$ norm on that interval. 
We state this precisely next. If $|I|$ is sufficiently large, 
then there exists some $t_{0} \in I$ such that
\begin{equation}
\| \int_{0}^{t_{0}} S(t, \tau)(0, u^{5}) d\tau \|_{L_{t,x}^{8}([t_{0}, \infty) \times \mathbb{R}^{3})} \ll 1.
\end{equation}
This says that, essentially, on any sufficiently long interval, we can find a ``quiet time'' $t_0$ after which the entire nonlinear history of the solution has a small effect. This is proven by combining two deep estimates: a dispersive estimate and an interaction Morawetz estimate. Partitioning $[0,t_0]$ into two intervals, one called the ``recent past'' and the other the ``remote past'', the dispersive estimate makes the remote past small by choosing $T$ large, while the interaction Morawetz estimate guarantees that on a long interval, one can find a $t_0$ where the recent past is small.

We apply the propagator $S(t, t_0)$ to the Duhamel representation of $(u(t_0), u_t(t_0))$ and use the group property of the flow:
\begin{align*}
S(t, t_0)(u(t_0), u_t(t_0)) &= S(t, t_0)\left( S(t_0)(u_0, u_1) + \int_0^{t_0} S(t_0,\tau)(0, u^5) d\tau \right) \\
&= S(t)(u_0, u_1) + \int_0^{t_0} S(t,\tau)(0, u^5) d\tau.
\end{align*}
The propagator identity $S(t, t_0)(u(t_0), u_t(t_0)) = S(t)(u_0, u_1) + \int_0^{t_0} S(t,\tau)(0, u^5) d\tau$ allows us to control the $L^8_{t,x}$ norm of the left-hand side by two terms. We use the triangle inequality for the $L^8_{t,x}$ norm. In the $L^8_{t,x}$ norm, the first term is small by the definition of the interval partition, while the second is rendered small by the Morawetz and dispersive estimates that justify the choice of $t_0$. %

Thus,
\begin{equation}
\| S(t, t_{0})(u(t_{0}), u_{t}(t_{0})) \|_{L_{t,x}^{8}([t_{0}, \sup(I)) \times \mathbb{R}^{3})} = \varepsilon \ll 1.
\end{equation}
By standard perturbation arguments, \begin{equation}\label{eq:L_eight}
\| u \|_{L_{t,x}^{8}([t_{0}, \sup(I)) \times \mathbb{R}^{3})} \lesssim 1.
\end{equation}
\begin{proof}[Proof of \cref{eq:L_eight}]
By the Strichartz estimate in \cite{MetcalfeTataruParametrices} for the variable-coefficient wave equation,
\begin{equation}
\| S(t, t_{0}) (u(t_{0}), u_{t}(t_{0})) \|_{L_{t}^{4} L_{x}^{12}} \lesssim E(t_{0})^{1/2} \lesssim E(0)^{1/2}.
\end{equation}
This estimate holds under our metric assumptions (as established by \cite{MetcalfeTataruParametrices}). 

Thus we have
\[
\|S(t,t_0)(u(t_0),u_t(t_0))\|_{L^8_{t,x}([t_0,\sup(I))\times \mathbb R^3)}=\varepsilon\ll1,
\qquad
\|S(t,t_0)(u(t_0),u_t(t_0))\|_{L^4_tL^{12}_x}\lesssim E(t_0)^{1/2}.
\]
Interpolating between $(8,8)$ and $(4,12)$ with $\theta=\frac35$ yields
\[
\|S(t,t_0)(u(t_0),u_t(t_0))\|_{L^5_tL^{10}_x}
 \lesssim 
\varepsilon^{2/5}\,E(t_0)^{3/10}.
\]
By inhomogeneous Strichartz and Hölder,
\[
\|u\|_{L^5_tL^{10}_x}
 \lesssim 
\varepsilon^{2/5}E(t_0)^{3/10}
 + 
\||u|^4u\|_{L^{5/4}_tL^{10/9}_x}
 = 
\varepsilon^{2/5}E(t_0)^{3/10}
 + 
\|u\|_{L^5_tL^{10}_x}^{5}.
\]
Similarly, using $(q,r)=(8,8)$ and $(4,12)$ for the output,
\[
\|u\|_{L^8_{t,x}} \lesssim \varepsilon + \|u\|_{L^5_tL^{10}_x}^5,
\qquad
\|u\|_{L^4_tL^{12}_x} \lesssim E(t_0)^{1/2} + \|u\|_{L^5_tL^{10}_x}^5.
\]
If $\varepsilon^{2/5}E(t_0)^{3/10}$ is sufficiently small, a standard bootstrap gives
\[
\|u\|_{L^5_tL^{10}_x}\lesssim \varepsilon^{2/5}E(t_0)^{3/10},
\]
and hence
\[
\|u\|_{L^8_{t,x}}\lesssim \varepsilon + E(t_0)^{3/2}\varepsilon^{2},
\qquad
\|u\|_{L^4_tL^{12}_x}\lesssim E(t_0)^{1/2} + E(t_0)^{3/2}\varepsilon^{2}.
\]
\end{proof}
This proves that \[
\|u\|_{L^8_{t,x}(I\times \R^3)} \lesssim_{E,A}1.
\]
Since there are finitely many intervals $I$, this proves \cref{thm:main}.

\section{Preliminaries, and the remote past}\label{sec:mainproof}
We first partition the time axis $\mathbb{R}$ into a finite number of intervals on which the contribution from the linear evolution is small. The existence of such a partition is a standard consequence of the finiteness of the global Strichartz norm. We state and prove this in the following proposition.

\begin{proposition}[Finite decomposition via Strichartz estimates]\label{prop:finite_decomp}
Let $v(t,x) = S(t)(u_0, u_1)$ be the solution to the linear homogeneous wave equation $\Box_g v = 0$ with initial data $(u_0, u_1)$ satisfying the energy bound \cref{1.5}. Let $C(E)$ be the constant from the global variable-coefficient Strichartz estimate, such that
\begin{equation} \label{eq:global_strichartz_bound}
    \|v\|_{L^8_{t,x}(\mathbb{R} \times \mathbb{R}^3)} \le C(E).
\end{equation}
Then for any $\eta > 0$, the time axis $\mathbb{R}$ can be decomposed into a finite number of intervals, $M = M(E, \eta)$,
\[
    \mathbb{R} = \bigcup_{k=1}^M I_k,
\]
such that on each interval, the solution satisfies
\[
    \|v\|_{L^8_{t,x}(I_k \times \mathbb{R}^3)} \le \eta.
\]
Furthermore, the number of intervals is explicitly bounded by $M = \lceil C(E)^8 \eta^{-8} \rceil$.
\end{proposition}

\begin{proof}
We provide the standard argument, which proceeds by partitioning the cumulative Strichartz integral. Let $B = C(E)$ be the bound from the global Strichartz estimate \eqref{eq:global_strichartz_bound}. We define the cumulative distribution function for the norm as
\[
F(t) := \int_{-\infty}^{t} \|v(\tau, \cdot)\|_{L^8_x}^8 \, d\tau.
\]
The function $F: \mathbb{R} \to [0, B^8]$ is continuous and non-decreasing, with $\lim_{t \to -\infty} F(t) = 0$ and $\lim_{t \to \infty} F(t) = \|v\|_{L^8_{t,x}}^8 \le B^8$.

Given $\eta > 0$, we define the number of intervals as $M := \lceil B^8 / \eta^8 \rceil$. We partition the range $[0, B^8]$ by setting $y_k = k \eta^8$ for $k = 0, 1, \dots, M-1$, and $y_M = B^8$. By the continuity of $F$, we may choose a sequence of time points $-\infty =: t_0 \le t_1 \le \dots \le t_{M-1} \le t_M := \infty$ such that $F(t_k) = y_k$ for $k=1, \dots, M-1$.

This sequence partitions the time axis $\mathbb{R}$ into $M$ intervals $I_k = [t_{k-1}, t_k]$ for $k = 1, \dots, M$. For each interval in this partition, we have
\begin{align*}
\|v\|_{L^8_{t,x}(I_k \times \mathbb{R}^3)}^8 &= \int_{t_{k-1}}^{t_k} \|v(\tau, \cdot)\|_{L^8_x}^8 \, d\tau \\
&= F(t_k) - F(t_{k-1}).
\end{align*}
By construction, the increment $F(t_k) - F(t_{k-1})$ is at most $\eta^8$ for each $k$. Thus $\|v\|_{L^8_{t,x}(I_k \times \mathbb{R}^3)} \le \eta$ for each $k=1, \dots, M$.
\end{proof}

With this partition in hand, the proof of \cref{thm:main} reduces to establishing a bound for $\|u\|_{L^8_{t,x}(I_k)}$ that is uniform in $k$. Let us fix one such interval $I := I_k$, on which we know $\|S(t)(u_0, u_1)\|_{L^8_{t,x}(I)} \le \eta$. The argument then finds a suitable time $t_0 \in I$ from which to control the nonlinear evolution. This is achieved by decomposing the Duhamel formula. For $t \in I$, we write
\begin{equation} \label{eq:duhamel_decomposition}
u(t) = S(t)(u_0, u_1) + \int_{0}^{t} S(t, \tau)(0, u^5(\tau)) \, d\tau.
\end{equation}
The subsequent propositions present an argument showing that for any sufficiently long interval $I$, there must exist a time $t_0 \in I$ where the contribution of the nonlinearity from the remote past ($\tau \ll t_0$) is small. This allows us to treat the solution on $[t_0, \sup(I))$ as a small perturbation of a free wave, yielding the necessary uniform bound on $\|u\|_{L^8(I)}$.

\subsection{The remote past}

\begin{proposition}[Negligibility/smallness of the remote-past contribution] \label{prop:negligible_remote_past}
Assume the hypotheses of \cref{thm:main}. Let $\eta > 0$ be given. There exists $T = T(E, A, \eta) > 0$ sufficiently large such that for any $t_0 >T$,
\begin{equation}
\left\| \int_0^{t_0 - T} S(\cdot,\tau)(0, u^5(\tau)) \, d\tau \right\|_{L^8_{t,x}([t_0, \infty) \times \R^3)} \le \eta.
\end{equation}
\end{proposition}

\begin{proof}
Our proof here depends on an $L^\infty$ bound whose proof we defer to \cref{sec:a_dispersive_estimate}. Here, we prove this proposition assuming that it holds.

Let $\mathcal{N}_{\mathrm{far}}$ equal the integral under investigation. We use the interpolation inequality
\begin{equation*}
    \| \mathcal{N}_{\mathrm{far}} \|_{L^8_{t,x}} \le \| \mathcal{N}_{\mathrm{far}} \|_{L^\infty_{t,x}}^{1/2} \cdot \| \mathcal{N}_{\mathrm{far}} \|_{L^4_{t,x}}^{1/2}.
\end{equation*}
We show the $L^\infty_{t,x}$ norm can be made arbitrarily small by choosing $T$ large (it is an $o_T(1)$ term), while the $L^4_{t,x}$ norm is bounded by a constant that depends on the a priori bounds on the solution, but is uniform in $t_0$ and $T$. The $L^\infty$ bound is proved in \cref{sec:a_dispersive_estimate}, and we prove the $L^4$ bound here.
For times $t \ge t_0 - T$, the function $\mathcal{N}_{\mathrm{far}}(t,x)$ is itself a solution to the homogeneous equation $P(\mathcal{N}_{\mathrm{far}}) = 0$. Its ``initial data'' at time $t_0 - T$ is given by
\begin{equation*}
    (\mathcal{N}_{\mathrm{far}}(t_0-T), \partial_t\mathcal{N}_{\mathrm{far}}(t_0-T)).
\end{equation*}
By evaluating the Duhamel formula for the full solution $u$, this data can be identified as $(u(t_0-T) - S(t_0-T)u[0], u_t(t_0-T) - \partial_t S(t_0-T)u[0])$.

The variable-coefficient Strichartz estimates for the pair $(p,q)=(4,4)$ give
\begin{equation*}
    \| \mathcal{N}_{\mathrm{far}} \|_{L^4_{t,x}([t_0-T, \infty))} \lesssim \| (\mathcal{N}_{\mathrm{far}}(t_0-T), \partial_t\mathcal{N}_{\mathrm{far}}(t_0-T)) \|_{\dot{H}^{1/2} \times \dot{H}^{-1/2}}.
\end{equation*}
By the triangle inequality, this Sobolev norm is bounded by
\begin{equation*}
    \|(u(t_0-T), u_t(t_0-T))\|_{\dot{H}^{1/2} \times \dot{H}^{-1/2}} + \|(u_0, u_1)\|_{\dot{H}^{1/2} \times \dot{H}^{-1/2}}.
\end{equation*}
The a priori bounds \eqref{assumption_to_be_improved} on $u$ and $u_t$ allow us to bound these $\dot{H}^{1/2} \times \dot{H}^{-1/2}$ norms by a constant $C(A)$ via interpolation. (In fact, a lower level of regularity than the higher degree norm in \cref{assumption_to_be_improved} is sufficient.) This bound is uniform in $t_0$ and $T$. So, $\| \mathcal{N}_{\mathrm{far}} \|_{L^4_{t,x}}$ is uniformly bounded by a constant depending only on $A$.
\end{proof}

\section{The recent past, and an interaction Morawetz estimate}\label{sec:interaction_Morawetz}

In this section, we adapt the interaction Morawetz estimate to the variable-coefficient operator $P$ and the equation \cref{eq:IVP} to establish a key global spacetime bound for the solution $u$. The interaction Morawetz estimate we will prove this section exists purely to prove \cref{prop:negligible_recent_nonlinear}. We just want to show that the solution must ``spread out'' after a sufficiently large time. By finite propagation speed, this means that the Strichartz norm is small for a long time, and by the time the solution can get back together, the dispersive estimate of \cref{sec:a_dispersive_estimate} takes over, proving scattering. This (interaction Morawetz) argument is similar to the interaction Morawetz argument in \cite{dodson-murphy2018}.

In the interaction Morawetz calculation, the main goal is to compute the time derivative of a certain functional $M_R(t)$. The key steps are to identify the main positive-definite terms that provide control, and to categorise the various error terms arising from the variable coefficients which are then bounded using integrated local energy decay (see \cref{thm:ILED0} for the simplest version of ILED, which is taken from \cite[Theorem 4.1]{Looi_Tohaneanu_2021}, and \cref{thm:high-order-ILED} for the higher-order ILED statements).

\begin{proposition}[Smallness of the recent-past contribution] \label{prop:negligible_recent_nonlinear}
Assume the hypotheses of \cref{thm:main} (in particular, the energy bound \cref{1.5} and the a priori bounds \cref{assumption_to_be_improved} hold), and let $T > 0$ be fixed.
For every $\eta > 0$, there exists a length threshold $L = L(\eta, E, A, T, g) > 0$ such that for any interval $I \subset \mathbb{R}$, if  $|I| > L$, then there exists a $t_0 \in I$ with
\begin{equation} \label{eq:local_nonlinear_small_eta}
\left\| \int_{t_{0} - T}^{t_{0}} S(t, \tau)(0, u^{5}(\tau)) d\tau \right\|_{L_{t,x}^{8}(\mathbb{R} \times \mathbb{R}^{3})} < \eta.
\end{equation}
\end{proposition}
We do not find the length threshold constant $L$ explicitly.

Suppose $h$ satisfies \cref{eq:hyp-a,eq:hyp-b,eq:hyp-c,eq:hyp-d}. 
The purpose of this interaction Morawetz estimate is to build a space-localized functional $M_R(t)$ whose time derivative contains a positive term dominating the density $|u|^{6}\langle t-|x|\rangle^{-1-\gamma}$.  
      Errors introduced by the coefficients are controlled using the decay assumptions \cref{eq:hyp} together with the local energy decay of \cite[Theorem 4.1]{Looi_Tohaneanu_2021}.

Choose $\phi \in C_{0}^{\infty}(\mathbb{R}^{3})$ that is supported on $|x| \leq 2$, $\phi(x) = 1$ on $|x| \leq 1$, $\phi(x) \geq 0$. Let $M_{R}(t)$ denote the Morawetz potential
\begin{equation}\label{2.2}
M_{R}(t) = \int e(t, y) \phi(\frac{x - y}{R}) (x - y) \cdot \langle u_{t}, \nabla u \rangle dx dy + \int e(t, y) \phi(\frac{x - y}{R}) \langle u_{t}, u \rangle ds dy,
\end{equation}
where $R > 0$ is a fixed constant and
\begin{equation}\label{2.3}
e(t, y) = -\frac{1}{2 g^{00}} g^{ij} (\partial_{i} u)(\partial_{j} u) + \frac{1}{2} u_{t}(t, y)^{2} + \frac{1}{6} u(t, y)^{6}.
\end{equation}
Then compute
\begin{equation}\label{2.4}
\frac{d}{dt} M_{R}(t) = \int e(t, y) \phi(\frac{x - y}{R}) (x - y) \cdot \langle \nabla u_{t}, u_{t} \rangle dx dy + \int e(t, y) \phi(\frac{x - y}{R}) \langle u_{t}, u_{t} \rangle dx dy
\end{equation}
\begin{equation}\label{2.5}
+ \int e(t, y) \phi(\frac{x - y}{R}) (x - y) \cdot \langle u_{tt}, \nabla u \rangle dx dy + \int e(t, y) \phi(\frac{x - y}{R}) \langle u_{tt}, u \rangle dx dy
\end{equation}
\begin{equation}\label{2.6}
+ \int e_{t}(t, y) \phi(\frac{x - y}{R}) (x - y) \cdot \langle \nabla u, u_{t} \rangle dx dy + \int e_{t}(t, y) \phi(\frac{x - y}{R}) \langle u_{t}, u \rangle dx dy.
\end{equation}

Integrating by parts,
\begin{equation}\label{2.7}
(\ref{2.4}) = -\frac{1}{2} \int e(t, y) \phi(\frac{x - y}{R}) |u_{t}|^{2} - \frac{1}{2} \int e(t, y) \phi'(\frac{x - y}{R}) \frac{|x - y|^{2}}{R} |u_{t}|^{2}.
\end{equation}
Now then, by \cref{eq:IVP},
\begin{equation}\label{2.8}
\aligned
u_{tt} = \frac{1}{g^{00}} g^{00} \partial_{0} \partial_{0} u = \frac{1}{g^{00}} \partial_{0}(g^{00} \partial_{0} u) - \frac{1}{g^{00}} \partial_{0}(g^{00}) \partial_{0} u \\
= \frac{1}{g^{00}} \partial_{\alpha}(g^{\alpha \beta} \partial_{\beta} u) - \frac{1}{g^{00}} \partial_{0}(g^{00}) \partial_{0} u - \frac{1}{g^{00}} \partial_{0}(g^{0j} \partial_{j} u) - \frac{1}{g^{00}} \partial_{i}(g^{i0} \partial_{0} u) - \frac{1}{g^{00}} \partial_{i} (g^{ij} \partial_{j} u) \\
= - \frac{1}{g^{00}} \partial_{0}(g^{00}) \partial_{0} u + \frac{1}{g^{00}} u^{5} - \frac{1}{g^{00}} \partial_{i} (g^{ij} \partial_{j} u) - \frac{1}{g^{00}} \partial_{0}(g^{0j} \partial_{j} u) - \frac{1}{g^{00}} \partial_{i}(g^{i0} \partial_{0} u).
\endaligned
\end{equation}
Plugging $(\ref{2.8})$ into $(\ref{2.5})$, observe that from \cref{thm:ILED0}, originating from Theorem 4.1 in \cite{Looi_Tohaneanu_2021}, if $u$ solves \cref{eq:IVP} and $h$ satisfies \cref{eq:hyp-a,eq:hyp-b,eq:hyp-c}, then
\begin{equation}\label{2.9}
\int_{\mathbb{R}} \int_{\mathbb{R}^{3}} \frac{|\nabla_{t,x} u|^{2}}{\langle x \rangle^{1 + \gamma}} + \frac{u^{2}}{\langle x \rangle^{3 + \gamma}} dx dt \lesssim E.
\end{equation}

Therefore, since $g^{00} = -1 + h^{00}$, \cref{thm:ILED0} and \cref{eq:hyp-c} imply
\begin{equation}\label{2.10}
\aligned
\int_{\mathbb{R}} \int e(t, y) \phi(\frac{x - y}{R}) (x - y) \cdot \langle -\frac{1}{g^{00}} \partial_{0}(g^{00}) \partial_{0} u, \nabla u\rangle dx dy dt \\ + \int_{\mathbb{R}} \int e(t, y) \phi(\frac{x - y}{R}) \langle -\frac{1}{g^{00}} \partial_{0}(g^{00}) \partial_{0} u, u \rangle dx dy dt \lesssim E^{2} R.
\endaligned
\end{equation}
Next, integrating by parts,
\begin{equation}\label{2.11}
\aligned
\int e(t, y) \phi(\frac{x - y}{R}) (x - y) \cdot \langle \frac{1}{g^{00}} u^{5}, \nabla u \rangle dx dy + \int e(t, y) \phi(\frac{x - y}{R}) \langle \frac{1}{g^{00}} u^{5}, u \rangle dx dy \\
= \frac{1}{2} \int e(t, y) \phi(\frac{x - y}{R}) \frac{1}{g^{00}(t, x)} u(t,x)^{6} dx dy - \frac{1}{6} \int e(t, y) \phi'(\frac{x - y}{R}) \frac{|x - y|}{R} \frac{1}{g^{00}(t,x)} u(t,x)^{6} dx dy \\
- \frac{1}{6} \int e(t, y) \phi(\frac{x - y}{R})(x - y) \cdot \nabla (\frac{1}{g(t,x)^{00}}) u(t, x)^{6} dx dy.
\endaligned
\end{equation}
Since $g^{00} = -1 + h^{00}$, then by \cref{thm:ILED0} and \cref{eq:hyp-c},
\begin{equation}\label{2.12}
\int_{\mathbb{R}} - \frac{1}{6} \int e(t, y) \phi(\frac{x - y}{R})(x - y) \cdot \nabla (\frac{1}{g(t,x)^{00}}) u(t, x)^{6} dx dy dt \lesssim E^{2} R.
\end{equation}
Next, integrating by parts,
\begin{equation}\label{2.13}
\aligned
-\int e(t, y) \phi(\frac{x - y}{R})(x - y) \cdot \langle \nabla u, \frac{1}{g^{00}} \partial_{i}(g^{ij} \partial_{j} u) \rangle dx dy - \int e(t, y) \phi(\frac{x - y}{R}) \langle u, \frac{1}{g^{00}} \partial_{i}(g^{ij} \partial_{j} u) \rangle dx dy \\
= \int e(t, y) \phi(\frac{x - y}{R})(x - y) \cdot \langle \nabla \partial_{i} u, \frac{1}{g^{00}} g^{ij} \partial_{j} u \rangle dx dy + \int e(t, y) \phi(\frac{x - y}{R}) \langle \partial_{i} u, \frac{1}{g^{00}} g^{ij} \partial_{j} u \rangle dx dy \\
+ \int e(t, y) \phi(\frac{x - y}{R}) \langle \partial_{i} u, \frac{1}{g^{00}} g^{ij} \partial_{j} u \rangle dx dy + \int e(t, y) \phi(\frac{x - y}{R}) \langle u, \partial_{i}(\frac{1}{g^{00}}) g^{ij} \partial_{j} u \rangle dx dy \\
+ \int e(t, y) \phi(\frac{x - y}{R})(x - y) \cdot \langle \nabla u, \partial_{i}(\frac{1}{g^{00}}) g^{ij} \partial_{j} u \rangle dx dy + \int e(t, y) \phi'(\frac{x - y}{R}) \frac{(x - y)_{i}}{|x - y| R} \langle g^{ij} \partial_{j} u, u \rangle dx dy \\
+ \int e(t, y) \phi'(\frac{x - y}{R}) \frac{(x - y)_{k} (x - y)_{i}}{|x - y| R} \langle \partial_{k} u, g^{ij} \partial_{j} u \rangle dx dy.
\endaligned
\end{equation}
Again, by \cref{thm:ILED0} and \cref{eq:hyp-a},
\begin{equation}\label{2.14}
\aligned
\int_{\mathbb{R}} \int e(t, y) \phi(\frac{x - y}{R}) \langle u, \partial_{i}(\frac{1}{g^{00}}) g^{ij} \partial_{j} u \rangle dx dy dt \\ + \int_{\mathbb{R}}  \int e(t, y) \phi(\frac{x - y}{R})(x - y) \cdot \langle \nabla u, \partial_{i}(\frac{1}{g^{00}}) g^{ij} \partial_{j} u \rangle dx dy dt \lesssim E^{2} R.
\endaligned
\end{equation}
Integrating by parts,
\begin{equation}\label{2.15}
\aligned
\int e(t, y) \phi(\frac{x - y}{R})(x - y) \cdot \langle \nabla \partial_{i} u, \frac{1}{g^{00}} g^{ij} \partial_{j} u \rangle dx dy + 2 \int e(t, y) \phi(\frac{x - y}{R}) \langle \partial_{i} u, \frac{1}{g^{00}} g^{ij} \partial_{j} u \rangle dx dy \\
= \frac{1}{2}  \int e(t, y) \phi(\frac{x - y}{R}) \langle \partial_{i} u, \frac{1}{g^{00}} g^{ij} \partial_{j} u \rangle dx dy - \frac{1}{2} \int e(t, y) \phi'(\frac{x - y}{R}) \frac{|x - y|}{R} \langle \partial_{i} u, \frac{1}{g^{00}} g^{ij} \partial_{j} u \rangle dx dy \\
- \frac{1}{2} \int e(t, y) \phi(\frac{x - y}{R}) (x - y) \cdot \nabla (\frac{g^{ij}}{g^{00}}) (\partial_{i} u)(\partial_{j} u) dx dy.
\endaligned
\end{equation}
Again by \cref{thm:ILED0} and \cref{eq:hyp-c},
\begin{equation}\label{2.16}
\int_{\mathbb{R}} \frac{1}{2} \int e(t, y) \phi(\frac{x - y}{R}) (x - y) \cdot \nabla (\frac{g^{ij}}{g^{00}}) (\partial_{i} u)(\partial_{j} u) dx dy dt \lesssim E^{2} R.
\end{equation}
Also, since $g^{ij} = 1 + h^{ij}$, and $g^{00} = -1 + h^{00}$, by the decay hypothesis $|h(t,x)| \lesssim \langle x \rangle^{-1-\gamma'}$ for some $\gamma' > 0$ (which is guaranteed by our stronger assumption \cref{eq:hyp-d}), and \cref{thm:ILED0}, for any $T \in \mathbb{R}$,
\begin{equation}\label{2.16.1}
\int_{0}^{T} \frac{1}{2}  \int e(t, y) \phi(\frac{x - y}{R}) \langle \partial_{i} u, \frac{1}{g^{00}} g^{ij} \partial_{j} u \rangle dx dy dt = -\int_{0}^{T} \frac{1}{2} e(t, y) \phi(\frac{x - y}{R}) |\nabla u|^{2} dx dy dt + O(E^{2} R).
\end{equation}
Again, by \cref{thm:ILED0} and the decay hypothesis $|h(t,x)| \lesssim \langle x \rangle^{-1-\gamma'}$ for some $\gamma' > 0$,
\begin{equation}\label{2.16.2}
\aligned
-\int_{0}^{T} \frac{1}{2} \int e(t, y) \phi(\frac{x - y}{R}) [|u_{t}|^{2} + |\nabla u|^{2}] dx dy dt \\ = -\int_{0}^{T} [\frac{1}{2} |u_{t}|^{2} + \frac{1}{2} |\nabla u|^{2} + \frac{1}{6} u^{6}] \phi(\frac{x - y}{R}) [\frac{1}{2} |u_{t}|^{2} + \frac{1}{2} |\nabla u|^{2}] dx dy dt + O(E^{2} R).
\endaligned
\end{equation}

Similarly,
\begin{equation}\label{2.16.3}
\begin{split}
\int \frac{1}{2} \int e(t, y) \phi\left(\frac{x - y}{R}\right) \frac{1}{g^{00}(t, x)} u(t,x)^{6} dx dy dt = \\
-\int \left[\frac{1}{2} |u_{t}|^{2} + \frac{1}{2} |\nabla u|^{2} + \frac{1}{6} u^{6}\right] \phi\left(\frac{x - y}{R}\right) \frac{1}{2} u(t,x)^{6} dx dy + O(E^{2} R).
\end{split}
\end{equation}
Since $g^{0j} = 0$, we can ignore the last two terms in $(\ref{2.8})$.\medskip

Now turn to $(\ref{2.6})$. By direct computation using the product rule,
\begin{equation}\label{2.17}
\partial_{t} e(t, y) = -\frac{1}{g^{00}} g^{jk} (\partial_{j} u)(\partial_{k} u_{t}) + u_{t} u_{tt} - \frac{1}{g^{00}} u_{t} u^{5} -\frac{1}{2} \partial_{t}(\frac{g^{jk}}{g^{00}}) (\partial_{j} u)(\partial_{k} u).
\end{equation}
As usual, by \cref{thm:ILED0} and \cref{eq:hyp-c},
\begin{equation}\label{2.18}
-\int_{\mathbb{R}} \int \partial_{t}(\frac{g^{jk}}{g^{00}}) (\partial_{j} u)(\partial_{k} u)(t, y) \phi(\frac{x - y}{R}) [(x - y) \cdot \langle u_{t}, \nabla u \rangle + \langle u_{t}, u \rangle] dx dy dt \lesssim E^{2} R.
\end{equation}
Using $(\ref{2.8})$ to compute $u_{tt}$,
\begin{equation}\label{2.19}
-\int_{\mathbb{R}} \int \frac{1}{g^{00}} u_{t} \partial_{0}(g^{00}) \partial_{0} u(t, y) \phi(\frac{x - y}{R}) [(x - y) \cdot \langle u_{t}, \nabla u \rangle + \langle u_{t}, u \rangle] dx dy dt \lesssim E^{2} R.
\end{equation}
Next, using the product rule,
\begin{equation}\label{2.20}
-\frac{1}{g^{00}} \partial_{i}(g^{ij} \partial_{j} u) u_{t} - \frac{1}{g^{00}} g^{ij}(\partial_{i} u)(\partial_{j} u_{t}) = -\frac{1}{g^{00}} \partial_{i}(g^{ij} \partial_{j} u u_{t}) = -\partial_{i}(\frac{1}{g^{00}} g^{ij} (\partial_{j} u) u_{t}) + g^{ij} (\partial_{j} u) u_{t} \partial_{i}(\frac{1}{g^{00}}).
\end{equation}
Then by \cref{thm:ILED0} and \cref{eq:hyp-c},
\begin{equation}\label{2.21}
\int_{\mathbb{R}} \int \partial_{i}(\frac{1}{g^{00}}) g^{ij}(\partial_{j} u) u_{t}(t, y) \phi(\frac{x - y}{R}) [(x - y) \cdot \langle u_{t}, \nabla u \rangle + \langle u_{t}, u \rangle] dx dy dt \lesssim E^{2} R.
\end{equation}
Next, integrating by parts,
\begin{equation}\label{2.22}
\aligned
-\int \partial_{i}(\frac{1}{g^{00}} g^{ij} (\partial_{j} u) u_{t}) \phi(\frac{x - y}{R}) [(x - y) \cdot \langle u_{t}, \nabla u \rangle + \langle u_{t}, u \rangle] dx dy \\ = -\int \frac{1}{g^{00}} g^{ij} (\partial_{j} u) u_{t} \phi(\frac{x - y}{R}) \langle u_{t}, \partial_{i} u \rangle dx dy \\
-\int \frac{1}{g^{00}} g^{ij} (\partial_{j} u) u_{t} \phi'(\frac{x - y}{R}) \frac{(x - y)_{i} (x - y)_{k}}{|x - y| R} \langle u_{t}, \partial_{k} u \rangle dx dy \\
-\int \frac{1}{g^{00}} g^{ij} (\partial_{j} u) u_{t} \phi'(\frac{x - y}{R}) \frac{(x - y)_{i}}{|x - y| R} \langle u_{t}, u \rangle dx dy
\endaligned
\end{equation}
Also,
\begin{equation}\label{2.23}
\int -\int \frac{1}{g^{00}} g^{ij} (\partial_{j} u) u_{t} \phi(\frac{x - y}{R}) \langle u_{t}, \partial_{i} u \rangle dx dy dt = \int \int \int (\partial_{i} u) u_{t} \phi(\frac{x - y}{R}) (\partial_{i} u) u_{t} dx dy dt + O(E^{2} R).
\end{equation}
Again, since $g^{0j} = 0$ for $j = 1, 2, 3$, we can perform the same computations as in section four of \cite{Dodson2024}. For exposition, we discuss these computations from that section 4 here, but we clarify that these rigorous arguments are available in \cite[Section 4]{Dodson2024}.

The preceding calculations have established that $M_R'(t)$ can be decomposed into a principal part, which is identical to the flat-space case, and error terms arising from the metric perturbation. Combining the estimates \eqref{2.7}, \eqref{2.11}, \eqref{2.15}, and \eqref{2.22}, and incorporating the bounds on the metric-dependent error terms from \eqref{2.10}, \eqref{2.12}, \eqref{2.14}, \eqref{2.16}, \eqref{2.18}, \eqref{2.19}, and \eqref{2.21}, we arrive at an expression of the form:
\begin{equation}
\frac{d}{dt} M_R(t) = \mathcal{I}_{\text{main}}(t) + \mathcal{I}_{\text{boundary}}(t) + \mathcal{E}(t),
\end{equation}
where $\mathcal{E}$ arises from $g$ not being equal to $m$. We have $|\int_{T_1}^{T_2} \mathcal{E}(t) dt| \lesssim E^2 R$ uniformly in time, and the main and boundary terms are those derived in \cite[Section 4]{Dodson2024}. Specifically, $\mathcal{I}_{\text{main}}(t)$ contains a positive-definite spacetime density. The term $\mathcal{I}_{\text{boundary}}(t)$ consists of integrals containing derivatives of the cutoff function (called $\phi'$ and $\phi''$ in \cite{Dodson2024}) which are supported on the annulus $R \le |x-y| \le 2R$ defined in \cite{Dodson2024}.

Following the argument in \cite{Dodson2024}, we now temporarily use notation from that paper. We average over the scale parameter $R$. Let $J \gg 1$ and $R_0 > 0$. We integrate in time from $a_j$ to $a_j + \mathcal{T}$ and average in $R$ over $[R_0, e^J R_0]$:
\begin{equation}
\frac{1}{J} \int_{a_j}^{a_j + \mathcal{T}} \int_{R_0}^{e^J R_0} \frac{1}{R} \frac{d}{dt} M_R(t) \,dR \,dt = \frac{1}{J} \int_{R_0}^{e^J R_0} \frac{1}{R} [M_R(a_j + \mathcal{T}) - M_R(a_j)] \,dR.
\end{equation}
The Morawetz potential is uniformly bounded by $|M_R(t)| \lesssim E^2 R$, so the right-hand side is bounded by $O(\frac{e^J R_0}{J} E^2)$. The averaging procedure converts the boundary terms $\mathcal{I}_{\text{boundary}}$ into spacetime integrals that can be absorbed into the main term for large $J$, while the error terms contribute at most $O(\frac{\mathcal{T}}{J} E^2)$. This procedure yields %
\begin{align}
\frac{1}{J} \int_{a_j}^{a_j + \mathcal{T}}   \int_{R_0}^{e^J R_0}   \frac{1}{R} \iint &   e(t,y) \phi\left(\frac{x-y}{R}\right) \left[ \frac{1}{2}(|u_t| - |\nabla u|)^2 + \frac{1}{6}u^6 \right] dx dy dR dt \nonumber \\
& \lesssim \frac{\mathcal{T}}{J} E^2 + \frac{e^J R_0}{J} E^2. \label{eq:morawetz_inequality}
\end{align}
By setting $\mathcal{T} = e^J R_0$, we ensure the right-hand side is $O(\frac{\mathcal{T}}{J} E^2)$. This inequality implies that for any long time interval $[a_j, a_j + \mathcal{T}]$, the spacetime density on the left must be small on average. By an averaging argument, there must exist a subinterval $[t_j - T, t_j] \subset [a_j, a_j + \mathcal{T}]$ where we gain control over a spacetime norm of $u$.

More precisely, as in \cite{Dodson2024}, the inequality \eqref{eq:morawetz_inequality} implies that for some $t_j \in [a_j+T, a_j+\mathcal{T}+T]$, we have
\begin{equation}
\int_{t_j-T}^{t_j} \sum_{k \in \mathbb{Z}^3} \left( \int \chi_k |u|^6 dx \right) \left( \int \chi_k (|\nabla u|^2 + u_t^2) dy \right) dt \lesssim \frac{T}{J} E^2,
\end{equation}
where $\chi_k$ are smooth partitions of unity localized to cubes of side length $R$. Interpolating this local $L^6_x$ bound with the a priori $L^\infty_t \dot{H}^2_x \le A$ bound allows us to deduce a smallness bound on an intermediate spacetime norm. Following the interpolation scheme in \cite[(4.43)-(4.49)]{Dodson2024}, we find
\begin{equation}\label{3.31}
\| u \|_{L_{t}^{27} L_{x}^{162/25}([t_{j} - T, t_{j}] \times \mathbb{R}^{3})} \lesssim \left(\frac{T}{J}\right)^{1/27} E^{1/9} A^{1/3}.
\end{equation}
Now then, by Strichartz estimates, as in Eq. (4.51) and the line after it in \cite{Dodson2024},
\begin{equation}
\| u \|_{L_{t}^{9/4} L_{x}^{54}([t_{j} - T, t_{j}] \times \mathbb{R}^{3})} \lesssim E^{1/2} + \| u \|_{L_{t}^{9/4} L_{x}^{54}}^{2} \| u \|_{L_{t}^{27} L_{x}^{162/25}}^{3}.
\end{equation}
Therefore,
\begin{equation}
\| u^{5} \|_{L_{t}^{1} L_{x}^{2}([t_{j} - T, t_{j}] \times \mathbb{R}^{3})} \lesssim \| u \|_{L_{t}^{9/4} L_{x}^{54}}^{2} \| u \|_{L_{t}^{27} L_{x}^{162/25}}^{3} \lesssim E \left(\frac{T}{J}\right)^{1/9} E^{1/3} A = T A^3 E^{1/3} J^{-1/9}
\end{equation} 
Then, by Strichartz estimates, since $\| u \|_{L_{t}^{9/4} L_{x}^{54}} \lesssim E^{1/2}$---which follows by 
\begin{equation}
\| \int_{t_{j} - T}^{t_{j}} S(t, \tau) (0, u^{5}) d\tau \|_{L_{t,x}^{8}} \lesssim  \| u^{5} \|_{L_{t}^{1} L_{x}^{2}} \lesssim \| u \|_{L_{t}^{9/4} L_{x}^{54}}^{2} \| u \|_{L_{t}^{27} L_{x}^{162/25}}^{3} \lesssim E \| u \|_{L_{t}^{27} L_{x}^{162/25}}^{3} \lesssim E \left(\frac{T}{J}\right)^{1/9} E^{1/3} A. 
\end{equation}
(There is no ``energy term'' in the Strichartz estimate since we are estimating the Duhamel term.) 
To ensure the smallness of the ``recent past'' nonlinear term, we require this norm to be $\le \delta$ where $\delta \sim (\epsilon/E)^{1/3}$, where $\epsilon$ is small. By choosing $J$ sufficiently large, %
we can achieve 
\[\| \int_{t_{j} - T}^{t_{j}} S(t, \tau) (0, u^{5}) d\tau \|_{L_{t,x}^{8}} \le \delta.\]
The choice of large $J$ was made in the choice of large $\mathcal T=e^J R_0$. 

Thus there exists $t_j$ for which the hypothesis of Proposition \ref{prop:negligible_recent_nonlinear} holds.%

We have just proved \cref{prop:negligible_recent_nonlinear}, which we restate here for the reader's convenience.
\begin{proposition}[Existence of a ``quiet'' time]\label{prop:quiet_time}
Under the hypotheses of \cref{thm:main}, for any $T > 0$, there exists a time $\mathcal{T} = \mathcal{T}(E, A, T,g)$ such that for any interval $I$ with $|I| \ge \mathcal{T}$, there is a time $t_0 \in I$ for which
\begin{equation}
\left\| \int_{t_{0} - T}^{t_{0}} S(t,\tau)(0, u^{5}(\tau)) d\tau \right\|_{L_{t,x}^{8}(\mathbb{R} \times \mathbb{R}^{3})} \ll 1.
\end{equation}
Thus if we have finitely many intervals $I_j$ (by the Strichartz estimate, as proved in \cref{prop:finite_decomp}), then we have finitely many of these quiet times $t_j$, with each $t_j$ associated to $I_j$.
\end{proposition}

\section{A dispersive estimate} \label{sec:a_dispersive_estimate}
We now turn to completing the proof of \cref{prop:negligible_remote_past}. To do so, we analyze the solution to the linear wave equation,
\begin{equation}\label{3.1}
Pu = 0, \qquad P = \partial_{\alpha} (g^{\alpha \beta} \partial_{\beta}) , \qquad u[0] = (u_{0}, u_{1}) \in \dot{H}^{1} \times L^{2}.
\end{equation}
Let $S(t, s)$ be the solution operator to $(\ref{3.1})$. That is,
\begin{equation}\label{3.2}
S(t, s)(f, g) = \phi(t, x),
\end{equation}
where $\phi(t, x)$ is the solution to
\begin{equation}\label{3.3}
P \phi = 0, \qquad P = \partial_{\alpha} (g^{\alpha \beta} \partial_{\beta}) \phi, \qquad \phi[s] = (f, g).
\end{equation}
We show that if
\begin{equation}\label{3.4}
	\aligned
\| \partial^{2} S(t, s)(f, g) \|_{L^{\infty}} \lesssim \frac{1}{\langle t - s\rangle} (\| (f, g) \|_{H^{\sigma} \times H^{\sigma - 1}} + \| (f, g) \|_{W^{4, 1} \times W^{3, 1}}), \\
\endaligned
\end{equation}
for some $\sigma$ sufficiently large, then for any $t_{0} \in \mathbb{R}$, $t > t_{0}$,
\begin{equation}\label{3.5}
\| \int_{0}^{t_{0} - T} S(t, s) (0, u^{5}) ds \|_{L^{\infty}} \lesssim_{A} o_{T}(1),
\end{equation}
where $A$ represents the a priori bounds on $u$,
\begin{equation}\label{3.6}
	\| (u, u_{t}) \|_{L_{t}^{\infty} H^{\sigma} \cap L^{2} \times H^{\sigma - 1} \cap H^{-1}} \leq A.
\end{equation}
The space $W^{k, 1}$ is the Sobolev space where up to $k$ derivatives of the function lie in $L^{1}$.

If $\phi$ solves $(\ref{3.3})$ then $\phi = U + v$, where $U$ solves
\begin{equation}\label{3.7}
\Box U = 0, \qquad U(s) = f, \qquad U_{t}(s) = g,
\end{equation}
and $v$ solves
\begin{equation}\label{3.8}
\Box v = \partial_{\alpha} (g^{\alpha \beta} \partial_{\beta}) \phi, \qquad v(s) = 0, \qquad v_{t}(s) = 0.
\end{equation}
By Duhamel's principle,
\begin{equation}\label{3.9}
\phi(t, x) = U(t, x) + \int_{s}^{t} S_{0}(t - \tau)(0, \Box v) d\tau.
\end{equation} where $S_{0}(t - \tau)$ is the solution operator to the Minkowski wave equation.

\paragraph{Proof of \cref{3.4}.}
Let $S(t,s)$ be the solution operator to $P=\partial_\alpha(g^{\alpha\beta}\partial_\beta)$, and let
\[
\phi(t)=S(t,s)(f,g),\qquad \tau:=t-s .
\]
Decompose $\phi=U+v$ by
\[
\square U=0,\ U[s]=f,\ U_t[s]=g,\qquad \square v=Q\phi:=\partial_\alpha(h^{\alpha\beta}\partial_\beta\phi),\ v[s]=v_t[s]=0 .
\]
The $U$-part obeys the desired bound by the $3\text{D}$ Kirchhoff/Poisson formula and differentiation of the free kernel:
\[
\|\partial^2 U(t)\|_{L^\infty_x}\ \lesssim\ \langle\tau\rangle^{-1}\,\|(f,g)\|_{W^{4, 1} \times W^{3, 1}} ,
\]
so it remains to control $v$. 

For the purpose of controlling $v$, the Klainerman--Sobolev inequality is not immediately applicable because the commutators $[P,L_i]$ with the Lorentz boosts produce coefficients of the form $t \partial_i h$, which become uncontrollable as $t \to \infty$ under our spatial decay assumptions on $h$ (in particular, $\partial_i h$ is \textit{not} assumed to temporally decay as $t \to \infty$). 

\begin{lemma}\label{lem:radial-Sobolev-annulus}
Let $f:\mathbb{R}^3\to\mathbb{R}$. Then for every $R>0$,
\[
\|f\|_{L^\infty(\{R<|x|<R+1\})}
\lesssim
R^{-1}\sum_{j=0}^1\sum_{|\alpha|\le2} \|\partial_r^j\Omega^\alpha f\|_{L^2(\{R-1<|x|<R+2\})}.
\]
\end{lemma}

\begin{proof}
Set $A_R=\{R-1<|x|<R+2\}$ and $I_R=[R,R+1]$. If $R\le 1$, then $A_R\subset B_3$ and the claim follows from the local embedding $H^2(B_3)\hookrightarrow L^\infty(B_3)$, since $R^{-1}\ge 1$. Hence assume $R\ge 1$.

Fix $r\in I_R$. By Sobolev on the unit sphere and scaling to $S_r$,
\begin{equation}\label{eq:sph-sob}
\|f\|_{L^\infty(S_r)}
\lesssim r^{-1}\sum_{|\alpha|\le 2}\|\Omega^\alpha f\|_{L^2(S_r)} .
\end{equation}
Write $F_\alpha(r):=\|\Omega^\alpha f(r\cdot)\|_{L^2(\mathbb S^2)}$, so $\|\Omega^\alpha f\|_{L^2(S_r)}=r\,F_\alpha(r)$. By the one-dimensional embedding $H^1([R-1,R+2])\hookrightarrow C^0([R-1,R+2])$,
\[
\sup_{r\in I_R}F_\alpha(r)
\lesssim \|F_\alpha\|_{L^2([R-1,R+2])}
+\|F_\alpha'\|_{L^2([R-1,R+2])},
\]
and $F_\alpha'(r)\le \|\partial_r(\Omega^\alpha f)(r\cdot)\|_{L^2(\mathbb S^2)}$.
Using the coarea formula and $r\sim R$ on $A_R$,
\[
\|F_\alpha\|_{L^2([R-1,R+2])}
+\|F_\alpha'\|_{L^2([R-1,R+2])}
\lesssim R^{-1}\Big(\|\Omega^\alpha f\|_{L^2(A_R)}
+\|\partial_r\Omega^\alpha f\|_{L^2(A_R)}\Big).
\]
Therefore,
\[
\sup_{r\in I_R}\|\Omega^\alpha f\|_{L^2(S_r)}
=\sup_{r\in I_R} r\,F_\alpha(r)
\lesssim \|\Omega^\alpha f\|_{L^2(A_R)}
+\|\partial_r\Omega^\alpha f\|_{L^2(A_R)} .
\]
\end{proof}

By \cref{lem:radial-Sobolev-annulus} and uniform energy bounds for the relevant vector-fields of $\partial^2S(t,s)(f,g)$---see \cref{sec:high-order-LE} for a proof---we have
\[
\|\partial^2  S(t, s)(f, g) \|_{L^\infty(R<|x|<R+1)}
 \lesssim_A 
R^{-1}.
\]
This implies that
\[
|\partial^2  S(t, s)(f, g) (x)|
 \le  C_A
\langle x \rangle^{-1}.
\]
By the $r$ to $t$ transfer from \cref{subsec:r_to_t}, we obtain \cref{3.4}, which contains norms of the data in the right-hand side. 

\emph{Pointer.} The high-order energy estimates of \cref{sec:high-order-LE} with $N=4$ provide the required vector-field control used here to obtain \cref{3.4}; see Remark~\ref{rem:derivative-bookkeeping}.

Similarly, again by \cref{lem:radial-Sobolev-annulus} and uniform energy bounds for the relevant vector-fields of $\partial S(t,s)(f,g)$, we have
\[
\|\partial  S(t, s)(f, g) \|_{L^\infty(R<|x|<R+1)}
 \lesssim_A 
R^{-1}
\] and hence also \[
|\partial  S(t, s)(f, g) (x)|
 \le  C_A
\langle t-s \rangle^{-1}.
\]

\subsection{The principle of $r$ to $t$ decay transfer}\label{subsec:r_to_t}

A challenge in establishing the global dispersive estimate \eqref{3.4} is the behavior of the solution in the interior region, $\{ |x| \ll t \}$. While local Sobolev inequalities can yield spatial decay estimates of the form $|\psi(t,x)| \lesssim \langle x \rangle^{-1}$, such bounds become trivial as $|x| \to 0$ and fail to provide the necessary temporal decay.

To overcome this, we employ a mechanism that converts spatial pointwise decay into temporal pointwise decay via integrated local energy decay (ILED) estimates. This was established in prior works; see \cite{MTT} for the context of stationary black hole spacetimes and, more directly, \cite[Section 8]{Looi2021} for a setting that generalizes the one considered here. 

We begin by defining the spacetime regions where the transfer occurs. For a dyadic time $T \ge 8$, let
\begin{align*}
    C_T^{\mathrm{int}} &:= \left\{ (t,x) \in \mathbb{R} \times \mathbb{R}^3 : \frac{1}{2}T \le t \le T, \ |x| \le \frac{t}{2} \right\}, \\
    \tilde{C}_T^{\mathrm{int}} &:= \left\{ (t,x) \in \mathbb{R} \times \mathbb{R}^3 : \frac{1}{4}T \le t \le T, \ |x| \le \frac{3t}{4} \right\}.
\end{align*}
The following lemma localizes the ILED estimate to these interior regions, providing the key tool for the transfer. Here, we write $\psi_{\le k}$ for the collection of functions generated by applying up to $k$ vector-fields from a standard set (translations, rotations, and the scaling vector-field, but \textit{not} Lorentz boosts) to $\psi$.

\begin{lemma}[Localized Interior ILED]\label{lem:localized-iled}
Let $\psi$ be a solution to $P\psi = F$ on the slab $[T/4, T] \times \mathbb{R}^3$. Then, for a sufficiently large integer $k$, its local energy norm is controlled by
\begin{equation}\label{eq:l2int}
    \|\psi\|_{LE^1(C_T^{\mathrm{int}})} \lesssim T^{-1} \|\langle r \rangle \psi_{\le k}\|_{LE^1(\tilde{C}_T^{\mathrm{int}})} + \|F_{\le k}\|_{LE^*(\tilde{C}_T^{\mathrm{int}})}.
\end{equation}
\end{lemma}

\begin{proof}[Proof sketch]
The proof applies the ILED estimate from \cref{sec:high-order-LE} to a spatially localized function. Let $\chi$ be a smooth cutoff function with $\chi \equiv 1$ on $C_T^{\mathrm{int}}$ and $\mathrm{supp}(\chi) \subset \tilde{C}_T^{\mathrm{int}}$, satisfying derivative bounds $|\partial^j \chi| \lesssim T^{-j}$. We consider the function $\psi^{\mathrm{int}} := \chi \psi$, which solves the equation
\[
    P\psi^{\mathrm{int}} = \chi F + [P, \chi] \psi.
\]
The ILED estimate for $\psi^{\mathrm{int}}$ on the time interval $[T/4, T]$ controls $\|\psi^{\mathrm{int}}\|_{LE^1}$ by its initial energy at $t=T/4$ and the $LE^*$ norms of the source terms. We observe that the commutator $[P, \chi]\psi$ is supported where $\nabla \chi \neq 0$ and involves derivatives of $\chi$. Its $LE^*$ norm can therefore be bounded by
\[
    \|([P, \chi]\psi)_{\le k}\|_{LE^*(\tilde{C}_T^{\mathrm{int}})} \lesssim T^{-1} \|\langle r \rangle \psi_{\le k}\|_{LE^1(\tilde{C}_T^{\mathrm{int}})}.
\]
The boundary energy term at $t=T/4$ is controlled by spacetime norms using a Morawetz-type argument based on the scaling vector-field, which introduces the same $T^{-1}$ scaling. For a detailed proof of this argument, see \cite{MTT} in the context of perturbations of the Schwarzschild and Kerr spacetimes, and \cite[Section 8]{Looi2021} for a more general setting that includes the one considered in this article. Combining these bounds yields the desired estimate \cref{eq:l2int}.
\end{proof}

This lemma provides the engine for the $r$-to-$t$ transfer.

\begin{proposition}\label{prop:r-to-t}
Let $\psi$ be a solution to the linear homogeneous equation $P\psi = 0$. Assume that in the interior region $\{|x| \le 3t/4\}$, $\psi$ and its vector-fields satisfy the spatial decay bound
\begin{equation}\label{eq:spatial-decay-hyp}
    |\psi_{\le k}(t,x)| \lesssim \langle x \rangle^{-1}
\end{equation}
for a sufficiently large integer $k$. Then, for $t \in [T/2, T]$, $\psi$ satisfies the temporal decay bound
\begin{equation}\label{eq:temporal-decay-conc}
    |\psi(t,x)| \lesssim \langle t \rangle^{-1}, \quad \text{for } |x| \le t/2.
\end{equation}
\end{proposition}

\begin{proof}[Proof sketch]
We use \cref{lem:localized-iled} with $F=0$. The hypothesis \cref{eq:spatial-decay-hyp} implies that the weighted energy density in the definition of the $LE^1$ norm, which involves terms like $\langle r \rangle^{-1-\gamma}|\nabla(\langle r \rangle \psi_{\le k})|^2$, decays like $\langle r \rangle^{-3-\gamma}$. Integrating this density over the spatial domain $\{|x| \le 3T/4\}$ shows that the norm of the weighted energy on any time slice is bounded uniformly in $T$. Integrating over the time interval $[T/4, T]$ yields
\[
    \|\langle r \rangle \psi_{\le k}\|_{LE^1(\tilde{C}_T^{\mathrm{int}})} \lesssim T^{1/2}.
\]
Substituting this into \cref{eq:l2int}, we obtain this bound on the local energy in the inner core:
\[
    \|\psi_{\le k}\|_{LE^1(C_T^{\mathrm{int}})} \lesssim T^{-1} \cdot T^{1/2} = T^{-1/2}.
\]
The proof concludes by combining this integrated energy bound into a pointwise estimate. A local Sobolev inequality, adapted to the geometry of the interior region (cf. \cite[Corollary 3.9]{MTT} or \cite[Corollary 3.2]{Looi2021}), yields
\[
    |\psi(t,x)| \lesssim T^{-1/2} \|\psi_{\le 3}\|_{LE^1(C_T^{\mathrm{int}})}, \quad \text{for } (t,x) \in C_T^{\mathrm{int}}.
\]
Combining these two estimates, we find that for any $t \in [T/2, T]$ and $|x| \le t/2$,
\[
    |\psi(t,x)| \lesssim T^{-1/2} \cdot (T^{-1/2}) = T^{-1}.
\]
Since $t \sim T$ in this region, this is equivalent to the desired bound $|\psi(t,x)| \lesssim \langle t \rangle^{-1}$.
\end{proof}

\paragraph{{Proof of \cref{3.10}}}
Plugging $(\ref{3.4})$ into $(\ref{3.8})$, by the estimates on the metric,
\begin{equation}\label{3.10}
|\partial_{\alpha} (h^{\alpha \beta} \partial_{\beta} ) \phi|(\tau, y) \lesssim_{A} \frac{1}{\langle \tau - s \rangle} \langle y \rangle^{-3 - \delta}.
\end{equation}
To see this, we begin by expanding $\partial_\alpha(h^{\alpha\beta}\partial_\beta\phi)
= h^{\alpha\beta}\partial_{\alpha\beta}\phi+(\partial_\alpha h^{\alpha\beta})\partial_\beta\phi$
and using \eqref{eq:hyp-d} together with the dispersive bound \eqref{3.4}. We have
\[
|h^{\alpha\beta}\partial_{\alpha\beta}\phi(\tau,y)|\lesssim
\langle y\rangle^{-3-\delta}\,\langle\tau-s\rangle^{-1}.
\]
\[
|(\partial_\alpha h^{\alpha\beta})\partial_\beta\phi(\tau,y)|
\lesssim \langle y\rangle^{-3-\delta}\,\langle\tau-s\rangle^{-1}.
\]
Combining the two contributions yields
\[
|\partial_\alpha(h^{\alpha\beta}\partial_\beta\phi)(\tau,y)|
\lesssim \langle\tau-s\rangle^{-1}\,\langle y\rangle^{-3-\delta},
\]
which is \eqref{3.10}.

\paragraph{Returning to the main proof} Making the one--dimensional reduction, if $|x| = r$,
\begin{equation}\label{3.11}
\begin{aligned}
|v(t,x)| = \left|\int_{s}^{t} S_{0}(t-\tau)\,(0,\Box v)(x)\,d\tau\right|
&\lesssim_{A} \frac{1}{2r} \int_{0}^{t-s} \frac{1}{\langle \tau \rangle}
  \int_{|r-(t-s-\tau)|}^{\,r+t-s-\tau} \rho\,\langle \rho \rangle^{-3-\delta}\, d\rho\, d\tau \\
&\lesssim_{A} \frac{1}{r} \int_{0}^{t-s}
  \langle r-(t-s-\tau) \rangle^{-1-\delta}\,\langle \tau \rangle^{-1}\, d\tau \\
&\lesssim_{A} \frac{1}{r}\Bigg[
  \int_{0}^{\,|r-t+s| - \frac{|(t-s)-r|}{2}}
      \langle \tau \rangle^{-1}\langle r-(t-s-\tau) \rangle^{-1-\delta}\, d\tau \\
&\qquad\quad+
  \int_{\,|r-t+s| - \frac{|(t-s)-r|}{2}}^{\,|r-t+s| + \frac{|(t-s)-r|}{2}}
      \langle \tau \rangle^{-1}\langle r-(t-s-\tau) \rangle^{-1-\delta}\, d\tau \\
&\qquad\quad+
  \int_{\,|r-t+s| + \frac{|(t-s)-r|}{2}}^{\infty}
      \langle \tau \rangle^{-1}\langle r-(t-s-\tau)\rangle^{-1-\delta}\, d\tau
\Bigg] \\
&\lesssim_{\delta} \frac{1}{r}\,\langle (t-s)-r \rangle^{-1}.
\end{aligned}
\end{equation}

The first line in \cref{3.11} follows directly from \cref{3.10}.

We explain the final line in \cref{3.11}. We have split the integral into three, with the middle one centred around $\tau=|r-t+s|$. 
Write $T:=t-s$, $a:=|T-r|$, and $\langle x\rangle=(1+x^{2})^{1/2}$. Since the integrand is
nonnegative we may extend $\int_{0}^{T}$ to $\int_{0}^{\infty}$. Note that
$|r-t+s|=|r-T|=a$, so the three cutoffs are $[0,\frac a2]$, $[\frac a2,\frac{3a}{2}]$, $[\frac{3a}{2},\infty)$.
Also $r-(T-\tau)=(r-T)+\tau=\varepsilon a+\tau$ with $\varepsilon=\mathrm{sgn}(r-T)\in\{\pm1\}$.

Thus we need to bound
\[
I:=\int_{0}^{\infty}\langle\tau\rangle^{-1}\,\langle\varepsilon a+\tau\rangle^{-1-\delta}\,d\tau
=I_1+I_2+I_3,
\]
with $\tau\in[0,\frac a2]$, $[\frac a2,\frac{3a}{2}]$, $[\frac{3a}{2},\infty)$ respectively.

\emph{(i) $\tau\in[0,\frac a2]$.} By $|\varepsilon a+\tau|\ge a-\tau\ge a/2$,
\[
I_1 \lesssim \langle a\rangle^{-1-\delta} \int_{0}^{a/2} \langle\tau\rangle^{-1}d\tau
 \lesssim_\delta \langle a\rangle^{-1},
\]
using $\log(1+a)\lesssim_\delta \langle a\rangle^{\delta}$.

\emph{(ii) $\tau\in[\frac a2,\frac{3a}{2}]$.} Here $\langle\tau\rangle\gtrsim\langle a\rangle$.
If $\varepsilon=+1$ then $\langle\varepsilon a+\tau\rangle=\langle a+\tau\rangle\gtrsim\langle a\rangle$, so
\[
I_2 \lesssim  (a)\,\langle a\rangle^{-2-\delta} \lesssim \langle a\rangle^{-1-\delta}\le \langle a\rangle^{-1}.
\]
If $\varepsilon=-1$, set $u=\tau-a\in[-a/2,a/2]$; then
\[
I_2=\int_{|u|\le a/2}\langle a+u\rangle^{-1}\langle u\rangle^{-1-\delta}\,du
 \lesssim \langle a\rangle^{-1} \int_{|u|\le a/2} \langle u\rangle^{-1-\delta}du
 \lesssim_\delta \langle a\rangle^{-1}.
\]

\emph{(iii) $\tau\in[\frac{3a}{2},\infty)$.} We have $\langle\tau\rangle\gtrsim\langle a\rangle$ and
$\langle\varepsilon a+\tau\rangle\gtrsim\langle\tau\rangle$, hence
\[
I_3 \lesssim \int_{3a/2}^{\infty}\langle\tau\rangle^{-2-\delta}d\tau
 \lesssim \langle a\rangle^{-1-\delta}\le \langle a\rangle^{-1}.
\]

Combining (i)-(iii),
\[
I \lesssim_\delta \langle a\rangle^{-1}=\langle (t-s)-r\rangle^{-1}.
\]
Multiplying by the prefactor $1/r$ yields
\[
\frac{1}{r}\Bigg[\cdots\Bigg] \lesssim_{\delta} \frac{1}{r}\,\langle (t-s)-r\rangle^{-1}.
\]

\vspace{.5cm}

We resume at the conclusion of \cref{3.11}. Making the $r$ to $t$ transfer, we have proved
\begin{equation}\label{3.12}
|\int_{s}^{t} S_{0}(t - \tau)(0, \Box v)(x_{0}) d\tau| \lesssim_{\delta, A} \frac{1}{\langle t - s \rangle} \langle (t - s) - r \rangle^{-1}.
\end{equation}
Let $\phi(t)$ be the solution with Cauchy data $\phi[s] = (\phi(s, \cdot), \partial_t\phi(s, \cdot))$ at time $s$. By definition of the full propagator $S(t,s)$,
\[
\phi(t) = S(t,s)\phi[s]
\]
By applying Duhamel's principle to the inhomogeneous equation $\Box\phi = F(\phi)$, we can express the same solution $\phi(t)$ as
\[
\phi(t) = S_0(t-s)\phi[s] + \int_s^t S_0(t-\tau)(0, F(\phi(\tau))) \,d\tau
\]
Equating these two expressions for $\phi(t)$ yields the identity
\[
S(t,s)\phi[s] = S_0(t-s)\phi[s] + \int_s^t S_0(t-\tau)(0, F(\phi(\tau))) \,d\tau
\]
Rearranging gives a formula for the operator difference:
\[
(S(t,s) - S_0(t-s))\phi[s] = \int_s^t S_0(t-\tau)(0, F(\phi(\tau))) \,d\tau
\]
Thus, by \cref{3.12},
\begin{equation}\label{3.13}
\begin{aligned}
\Bigl\| \int_{0}^{t_{0}-T}  \bigl[S(t,s)(0,u^{5}) - S_{0}(t-s)(0,u^{5})\bigr]\,ds \Bigr\|_{L^{\infty}}
&\lesssim_{\delta,A} \sup_{r\ge0}\int_{0}^{t_{0}-T} \frac{1}{\langle t-s\rangle}\,\langle (t-s)-r\rangle^{-1}\,ds \\
&\lesssim T^{-1+ \delta}
\end{aligned}
\end{equation}for any $\delta>0$. The second bound is proved in the following lemma.

\begin{lemma}
Let $t,t_0,T\in\mathbb{R}$ with $t>t_0>T>0$. For any $\delta\in(0,1)$,
\[
\sup_{r\ge 0}\int_{0}^{t_0-T}\frac{1}{\langle t-s\rangle\,\langle (t-s)-r\rangle}\,ds
\lesssim_\delta \langle T\rangle^{-1+\delta}.
\]
\end{lemma}
The idea of the proof is as follows. After setting $y=t-s$, the integration interval becomes $[t-t_0+T,t]\subset[T,\infty)$, which yields the uniform lower bound $\langle y\rangle\gtrsim \langle T\rangle$; applying Hölder with any exponent $>1$ turns the remaining factor $\langle y-r\rangle^{-1}$ into an $r$-uniform constant and gives the decay $\lesssim \langle T\rangle^{-1+\delta}$ with $\delta=\tfrac{1}{p}$.
\begin{proof}
Let $y=t-s$. Then $y\in[a,b]$ with $a=t-t_0+T\ge T$ and $b=t$. Fix $p>1$ and set $q=\frac{p}{p-1}$. By Hölder,
\[
\int_{a}^{b}\frac{dy}{\langle y\rangle\,\langle y-r\rangle}
\le
\Big(\int_{a}^{b}\langle y\rangle^{-p}\,dy\Big)^{1/p}
\Big(\int_{a}^{b}\langle y-r\rangle^{-q}\,dy\Big)^{1/q}.
\]
The second factor is uniformly bounded in $r$ by
\[
\Big(\int_{\mathbb{R}}\langle z\rangle^{-q}\,dz\Big)^{1/q}<\infty
\qquad(q>1).
\]
For the first factor, since $a\ge T$,
\[
\int_{a}^{\infty}\langle y\rangle^{-p}\,dy \lesssim_p \langle a\rangle^{1-p}
\lesssim_p \langle T\rangle^{1-p},
\]
hence
\[
\int_{a}^{b}\frac{dy}{\langle y\rangle\,\langle y-r\rangle}
\lesssim_p \langle T\rangle^{-(p-1)/p}.
\]
Taking the supremum in $r\ge0$ and writing $\delta=\frac{1}{p}\in(0,1)$ yields
\[
\sup_{r\ge0}\int_{0}^{t_0-T}\frac{ds}{\langle t-s\rangle\,\langle (t-s)-r\rangle}
\lesssim_\delta \langle T\rangle^{-1+\delta}.
\]
\end{proof}

\section{Uniform energy bounds and integrated local energy decay for high orders}\label{sec:high-order-LE}

This section collects the local energy decay and associated estimates that make the interaction Morawetz and dispersive arguments go through with uniform constants on asymptotically flat backgrounds.

Let $g^{\mu\nu}=m^{\mu\nu}+h^{\mu\nu}$ on $\mathbb{R}^{1+3}$ and $Pu=\partial_\mu(g^{\mu\nu}\partial_\nu u)$. We study strong solutions of
\begin{equation}\label{eq:main-nl}
Pu=u^5\quad\text{on }(T_1,T_2)\times\mathbb{R}^3.
\end{equation}
Write $\langle r\rangle=(1+|x|^2)^{1/2}$, fix $0<\gamma\ll1$, and set
\[
\|w\|_{LE^1[t_1,t_2]}^2:= \int_{t_1}^{t_2}   \int \frac{|\nabla_{t,x}w|^2}{\langle r\rangle^{1+\gamma}}+\frac{|w|^2}{\langle r\rangle^{3+\gamma}}\,dx\,dt,\qquad
\|F\|_{LE^*[t_1,t_2]}^2:= \int_{t_1}^{t_2}   \int \langle r\rangle\,|F|^2\,dx\,dt,
\]
and $E_N(t):=\sum_{|\alpha|\le N}\frac12\|\nabla_{t,x}\partial^\alpha u(t)\|_{L^2_x}^2$.

\begin{theorem}[ILED-0: integrated local energy decay and uniform energy bounds from Theorem 4.1 in \cite{Looi_Tohaneanu_2021}]\label{thm:ILED0}
Let $u$ solve
\[
\begin{cases}
P u = u^5 & (t,x)\in (0,\infty)\times\mathbb{R}^3,\qquad P=\partial_\alpha g^{\alpha\beta}\partial_\beta,\\[2pt]
u[0]\in \dot H^1\times L^2,
\end{cases}
\]
on a Lorentzian background $g=m+h$. Assume for some fixed $\gamma>0$ and sufficiently small $\epsilon>0$ that
\[
\partial_{t,x}^J h^{\alpha\beta}(t,x)\ \lesssim\ \epsilon\,\langle x\rangle^{-|J|-\gamma}\quad\text{for all multi-indices } J \text{ with }|J|\in \{0,1\}.
\]
Then for any $T_1<T_2$,
\begin{equation}\label{eq:ILED0}
\|u\|_{LE^1[T_1,T_2]}^2+E(T_2)\ \lesssim\ E(T_1),
\end{equation}
where
\[
E(t):=\int_{\{t\}\times\mathbb{R}^3}\Big(\frac12|\nabla_{t,x} u|^2+\frac16|u|^6\Big)\,dx,
\qquad
\|u\|_{LE^1[T_1,T_2]}^2
:=\iint_{[T_1,T_2]\times\mathbb{R}^3} 
\frac{|\nabla_{t,x} u|^2}{\langle r\rangle^{1+\gamma}}
+\frac{u^2}{\langle r\rangle^{3+\gamma}}
+\frac{|u|^6}{\langle r\rangle}\,dxdt.
\]
\end{theorem}

\begin{proof}[Proof sketch]
This is the standard Morawetz multiplier argument adapted to variable coefficients, and the full proof was published in \cite{Looi_Tohaneanu_2021}. 
Let
\[
b(r)=\sum_{j\ge0}2^{-j\gamma}\frac{r}{r+2^j},\qquad a(r)=\frac{b(r)}{r},
\]
and multiply $Pu=u^5$ by $a(r)u+b(r)\partial_r u+C\,\partial_t u$, with $C\gg1$.
Integrating by parts on $[T_1,T_2]\times\mathbb{R}^3$ yields
\[
E(T_2)
+\iint \Big[\frac12 b'(u_r^2+u_t^2)
+\Big(-\frac12 b'+\frac{b}{r}\Big)|\partial_\omega u|^2
-\frac12(\Delta a)\,u^2\Big]\,dxdt
\ \lesssim\ 
E(T_1)+\mathrm{Err},
\]
with 
\[
\mathrm{Err}\ \lesssim\ \big(\frac{|h|}{\langle x\rangle}+|\nabla_{t,x} h|\big)\Big(|\nabla_{t,x} u|^2+|\nabla_{t,x} u|\,\frac{|u|}{\langle x\rangle}\Big).
\]
The choice of $a,b$ gives (uniformly in $t$)
\[
b'\gtrsim \langle r\rangle^{-1-\gamma},\qquad 
-\Delta a\gtrsim \langle r\rangle^{-3-\gamma},\qquad 
\frac{2}{3}a-\frac{1}{6}b'\gtrsim \langle r\rangle^{-1},
\]
so the left-hand side controls $\|u\|_{LE^1[T_1,T_2]}^2+E(T_2)$. By the coefficient assumption
$|\partial^J h|\lesssim \epsilon\langle x\rangle^{-|J|-\gamma}$ with $|J|\le1$ and Cauchy-Schwarz,
\[
\mathrm{Err}\ \lesssim\ \epsilon \iint \Big(\frac{|\nabla_{t,x} u|^2}{\langle r\rangle^{1+\gamma}}
+\frac{u^2}{\langle r\rangle^{3+\gamma}}\Big)\,dxdt,
\]
which can be absorbed for $\epsilon$ small, giving \eqref{eq:ILED0} for classical solutions.
For strong solutions, approximate by classical data at time $T_1$, use local well-posedness and a partition of the time interval so that $\|u\|_{L^5_tL^{10}_x}$ is small on each subinterval, pass to the limit, and conclude \eqref{eq:ILED0}.
\end{proof}
Next, we state and partly prove the higher-order local energy decay and uniform energy bounds. These estimates require the high-order regularity condition \eqref{eq:hyp-c} from our main set of assumptions on the metric.

\begin{remark}[Derivative bookkeeping and the choice $N=4$]
\label{rem:derivative-bookkeeping}
This remark justifies the a priori $\dot H^5 \times \dot H^4$ regularity bound ($A$) required in Theorem~\ref{thm:main}. The bound arises from the maximum regularity needed across the different components of our proof, which do not stack. The principal requirements are dictated by the dispersive estimate, which we break down as follows:
\begin{enumerate}
    \item \textbf{Exterior region ($|x| > ct$):} Control in the exterior relies on an annular Sobolev inequality (Lemma~\ref{lem:radial-Sobolev-annulus}) applied to $\partial^2 u$. This step requires control of \emph{fifth-order derivatives} in $L^2$, which is provided by the high-order uniform energy bound. In terms of regularity required, this is the most demanding component of the entire proof.

    \item \textbf{Interior region ($|x| \ll t$):} Control in the interior is achieved via the $r \to t$ decay transfer (Proposition~\ref{prop:r-to-t}). This mechanism requires control of the integrated local energy decay up to \emph{third order}.

    \item \textbf{Morawetz argument:} The interaction Morawetz estimate is fundamentally an $\dot H^1$-level calculation. However, its application within the ``recent past'' argument critically relies on the high-regularity bound $A$ to make the output of the estimate useful via interpolation (see Eq.~\eqref{eq:morawetz_inequality} and the argument leading to \eqref{3.31}).
\end{enumerate}
The overall $\dot H^5$ requirement ($N=4$) is therefore dictated by the most demanding component, which is the exterior estimate.

This regularity requirement on the solution translates directly to the metric via the commutator estimates (Lemma~\ref{lem:comm-LEstar}). To control commutators involving $N$-th order derivatives of the solution, we need control of $(N+1)$-th order derivatives of the metric. Accordingly, our $N=4$ assumption on the solution necessitates the $|J|\le 5$ assumption on the metric in Assumption~\ref{eq:hyp-c}.
\end{remark}

We use the order 0 ILED as input (proved earlier in \cref{thm:ILED0}), as well as a Strichartz estimate. Thus if $Pw=F$ on $[t_1,t_2]$, then
\begin{align}
\|w\|_{LE^1[t_1,t_2]}^2+\sup_{t\in[t_1,t_2]}\|\nabla_{t,x}w(t)\|_{L^2_x}^2
&\lesssim \|\nabla_{t,x}w(t_1)\|_{L^2_x}^2+\|F\|_{LE^*[t_1,t_2]}^2,\label{eq:ILED0}\\
\|w\|_{L^5_tL^{10}_x([t_1,t_2])}
&\lesssim \|\nabla_{t,x}w(t_1)\|_{L^2_x}+\|F\|_{LE^*+L^{5/4}_tL^{10/9}_x([t_1,t_2])}.\label{eq:strichartz0}
\end{align}

\begin{lemma}[Commutator]\label{lem:comm-LEstar}
For $|\alpha|\le N$,
\begin{equation}\label{eq:comm-LEstar}
\|[P,\partial^\alpha]u\|_{LE^*[t_1,t_2]}
\ \lesssim\ \varepsilon\sum_{|\beta|\le|\alpha|}
\|\partial^\beta u\|_{LE^1[t_1,t_2]}.
\end{equation}
\end{lemma}

\begin{proof}[Proof sketch]
Expanding $P=\partial_\mu(g^{\mu\nu}\partial_\nu)$ and commuting, one term will involve an estimate of
$(\partial^{\sigma+1}g)\,\partial v$ in $LE^*$:
\[
\|(\partial^{\sigma+1}g)\,\partial v\|_{LE^*}^2
=\int\langle r\rangle\,|\partial^{\sigma+1}g|^2\,|\partial v|^2dxdt
\ \lesssim\ \varepsilon^2 \int \langle r\rangle^{-1-2\gamma}\,|\partial v|^2
\ \lesssim\ \varepsilon^2\,\|v\|_{LE^1}^2,
\]
using \cref{eq:hyp-c} and $\langle r\rangle^{-1-2\gamma}\le\langle r\rangle^{-(1+\gamma)}$. The other terms are not very different. Summing over $\sigma$ and components gives \eqref{eq:comm-LEstar}. This is proved in \cite[Section 4]{Looi2022}. 
\end{proof}

\begin{theorem}[High-order ILED and uniform energies, from Theorem 4.7 in \cite{Looi2022}]\label{thm:high-order-ILED}
Let $N \in \mathbb{N}$ and assume \eqref{eq:hyp-c} holds up to order $N$ with $\varepsilon>0$ sufficiently small (depending on $N,\gamma$). Let $u$ be a strong solution to \eqref{eq:main-nl} on $(T_1,T_2)$. Then
\begin{equation}\label{eq:main-high}
\sum_{|\alpha|\le N} \Bigl(\|\partial^\alpha u\|_{LE^1[T_1,T_2]}^2+\sup_{t\in[T_1,T_2]}\|\nabla_{t,x}\partial^\alpha u(t)\|_{L^2_x}^2\Bigr)
\ \lesssim\ \sum_{|\alpha|\le N}\|\nabla_{t,x}\partial^\alpha u(T_1)\|_{L^2_x}^2.
\end{equation}
with an implicit constant depending only on $N$ and $\gamma$ (and independent of $T_1<T_2$).
\end{theorem}

This theorem, and a more general version for vector-fields of $u$ instead of pure derivatives, is proved in \cite[Section 4]{Looi2022}. 

The passage to strong solutions follows by approximation.

\end{document}